\documentclass[a4paper,12pt,thmsa]{amsart}
\usepackage{amsmath, latexsym, amsfonts, amssymb, amsthm, amscd}
\usepackage{color}
\newtheorem{theorem}{Theorem}
\newtheorem{theo}{Theorem}
\newtheorem{corollary}{Corollary}
\newtheorem{proposition}{Proposition}
\newtheorem{lemma}{Lemma}
\newtheorem{rem}{Remark}
\newtheorem{defi}{Definition}
\newcommand{\p}{\Bbb{P}}
\newcommand{\px}{\Bbb{P}_x}
\newcommand{\e}{\Bbb{E}}

\newcommand{\ind}{\mbox{\rm 1\hspace{-0.04in}I}}

\newcommand{\ed}{\stackrel{(d)}{=}}

\title[Inversion, duality and Doob $h$-transforms for ssMp's]
{Inversion, duality and Doob $h$-transforms for self-similar Markov processes}

\author{L.~Alili \and L.~Chaumont \and P.~Graczyk \and T. \.Zak}
\address{L.~Alili -- Department of
Statistics, The University of Warwick, CV4 7AL, Coventry, UK.}
\email{L.Alili@warwick.ac.uk}
\address{L. Chaumont -- LAREMA UMR CNRS 6093, Universit\'e d'Angers, 2, Bd Lavoisier \\
Angers Cedex 01, 49045, France}
\email{loic.chaumont@univ-angers.fr}
\address{P.~Graczyk  -- LAREMA UMR CNRS 6093, Universit\'e d'Angers, 2, Bd Lavoisier \\
Angers Cedex 01, 49045, France}
\email{piotr.graczyk@univ-angers.fr}
\address{T. \.Zak  --  Faculty of Pure and Applied Mathematics, Wroc{\l}aw University of
Technology, Wybrze\.ze Wyspia\'nskiego 27, 50-370 Wroc{\l}aw, Poland.}
\email{tomasz.zak@pwr.edu.pl}


\date{\today}

\begin{document}

\begin{abstract}  We show that any $\mathbb{R}^d\setminus\{0\}$-valued self-similar Markov process $X$, with index $\alpha>0$  
can be represented as a path transformation of some Markov additive process (MAP) $(\theta,\xi)$ in $S_{d-1}\times\mathbb{R}$. 
This result extends the well known Lamperti transformation. Let us denote by $\widehat{X}$ the self-similar Markov process
which is obtained from the MAP $(\theta,-\xi)$ through this extended Lamperti transformation. Then we prove that $\widehat{X}$
is in weak duality with $X$, with respect to the measure $\pi(x/\|x\|)\|x\|^{\alpha-d}dx$, if and only if $(\theta,\xi)$ is reversible with 
respect to the measure $\pi(ds)dx$, where $\pi(ds)$ is some $\sigma$-finite measure on $S_{d-1}$ and $dx$ is the Lebesgue 
measure on $\mathbb{R}$. Besides, the dual process $\widehat{X}$ has the same law as the inversion  
$(X_{\gamma_t}/\|X_{\gamma_t}\|^2,t\ge0)$  of $X$, where $\gamma_t$ is the inverse of $t\mapsto\int_0^t\|X\|_s^{-2\alpha}\,ds$. 
These results allow us to obtain excessive functions for some classes of self-similar Markov processes such as stable L\'evy 
processes.
\end{abstract}

\maketitle

\section{Introduction}

There exist many ways to construct the three dimensional Bessel process from Brownian motion. It is generally defined as the 
strong solution of a stochastic differential equation driven by Brownian motion or as the norm of the three dimensional Brownian motion. 
It can also be obtained by conditioning Brownian motion to stay positive. Then there are several path transformations.  Let us focus on 
the following example.  

\begin{theo}[M. Yor, \cite{yo}] \label{6214}
Let $\{(B^0_t)_{t\ge0},\p_x\}$ and $\{(R_t)_{t\ge0},{\rm P}_x\}$, $x>0$, be respectively the standard Brownian motion absorbed at $0$ 
and the three dimensional Bessel process. Then $\{(R_t)_{t\ge0},{\rm P}_x\}$ can be constructed from $\{(B^0_t)_{t\ge0},\p_x\}$ through 
the following path transformation: 
\[\{(R_t)_{t\ge0},{\rm P}_x\}=\{(1/B^0_{\gamma_t})_{t\ge0},\p_{1/x}\},\]
where $\gamma_t=\inf\{s:\int_0^s\frac{du}{(B_u^0)^4}>t\}$.
\end{theo}      
\noindent This result was actually obtained in higher dimension in \cite{yo} where the law 
of the time changed inversion of $d$-dimensional Brownian motion is fully described. 
Recalling that three dimensional Bessel process is a Doob $h$-transform of Brownian motion absorbed at 0, the following result 
can be considered as a counterpart of Theorem \ref{6214} for isotropic stable L\'evy processes. 

\begin{theo}[K. Bogdan, T. \. Zak, \cite{bz}]\label{4503}
Let $\{(X_t)_{t\ge0},\p_x\}$, $x\in\mathbb{R}^d\setminus\{0\}$ be a $d$-dimensional, isotropic stable L\'evy process with index
$\alpha\in(0,2]$, which is absorbed at its first hitting time of $0$. Then the process
\begin{equation}\label{8422}
\{(X_{\gamma_t}/\|X\|^2_{\gamma_t})_{t\ge0},\p_{x/\|x\|^2}\}\,,
\end{equation}
where $\gamma_t=\inf\{s:\int_0^s\frac{du}{\|X\|_u^{2\alpha}}>t\}$, 
is the Doob $h$-transform of $X$ with respect to the positive harmonic function $x\mapsto\|x\|^{\alpha-d}$.
\end{theo}
\noindent When $d=1$ and $\alpha>1$, Yano \cite{ya} showed that the $h$-process which is involved in Theorem \ref{4503} can
be interpreted as the L\'evy process $\{(X_t)_{t\ge0},\p_x\}$, conditioned to avoid 0, see also Pant\'{\i} \cite{pa}. Then recently 
Kyprianou \cite{ky} proved that Theorem \ref{4503} is actually valid for any real valued stable L\'evy process. 

Comparing Theorems \ref{6214} and \ref{4503}, we notice that they are concerned with the same path transformation of 
some Markov process, and that the resulting Markov process can be obtained as a Doob $h$-tranform of the initial process. Then 
one is naturally tempted to look for a general principle which would allow us to prove an overall result in an appropriate framework. 
It clearly appears that the self-similarity property is essential in these path transformations. Therefore a first step in our 
approach was an  indepth study of the structure of self-similar Markov processes. This led us to an extension of the famous 
Lamperti representation. The latter is the object of the next section, see Theorem \ref{2461}, and represents one of our main 
results. It asserts that any self-similar Markov process absorbed at 0 can be represented as a time changed Markov additive 
process and actually provides a one-to-one relationship between these two classes of processes.

  Then Section 3 is devoted to the study of the time changed inversion (\ref{8422}) when $\{(X_t)_{t\ge0},\p_x\}$ is any self-similar 
Markov process absorbed at 0. Another important step in our reasoning is the characterisation, in Theorem \ref{5372}, of self-similar 
Markov processes $\{(X_t)_{t\ge0},\p_x\}$, which are in duality with the time changed inversion 
$\{(X_{\gamma_t}/\|X\|^2_{\gamma_t})_{t\ge0},\p_{x/\|x\|^2}\}$. We show that a necessary and sufficient condition for this to hold is that 
the underlying Markov additive process in the Lamperti representation satisfies a condition of reversibility. Some important classes 
of processes satisfying this condition are also described. The results of this section extend those obtained by Graversen and 
Vuolle-Apiala in \cite{gv}.
 
It remains to appeal to some link between duality and Doob $h$-transform. More specifically in Section \ref{7325}, we recover 
Theorems \ref{6214} and \ref{4503}, and Theorem 4 in \cite{ky} as consequences Theorem \ref{5372} in Section 3 and the simple
observation that if two Markov processes are in duality between themselves then it is also the case for their Doob $h$-transforms. 
This general principle actually applies to large classes of self-similar Markov processes and allows us to obtain excessive functions 
attached to them. We end the paper by reviewing the examples of contitioned stable L\'evy processes, free 
Bessel processes and Dunkl processes. Let us finally emphasize that another incentive for our work was the recent paper from
Alili, Graczyk and \.Zak \cite{agz}, where some relationships between inversions and $h$-processes are provided in the 
framework of diffusions.

\section{Lamperti representation of $\mathbb{R}^d\setminus\{0\}$-valued ssMp's}\label{sec1}

All Markov processes considered in this work are standard processes for which there is a reference measure. Let us first briefly recall 
these definitions from Section I.9 and Chapter V of \cite{bg}. A standard Markov process $Z=(Z_t)_{t\ge0}$ is a strong Markov process, 
with values in some state space $E_\delta=E\cup\{\delta\}$, where $E$ is a locally compact space with a countable base, $\delta$ is 
some isolated extra state and $E_\delta$ is endowed with its topological Borel $\sigma$-field. The process $Z$ is defined on some
completed, filtered probability space $(\Omega,\mathcal{F},(\mathcal{F}_t)_{t\ge0},(P_x)_{x\in E_\delta})$, where $P_x(Z_0=x)=1$, 
for all $x\in E_\delta$.The state $\delta$ is absorbing, that is $Z_t=\delta$, for all $t\ge\zeta(Z):=\inf\{t:Z_t=\delta\}$ and 
$\zeta(Z)$ will be called the lifetime of $Z$. The paths of $Z$ are assumed to be right continuous on $[0,\infty)$. Besides, 
they have left limits and are quasi-left continuous on $[0,\zeta)$. Finally, we assume that there is a reference measure, that 
is a $\sigma$-finite measure $\mu(dy)$ on $E$ such that for each $x\in E$, the potential measure 
$E_x(\int_0^\zeta \ind_{\{Z_t\in dy\}}\,dt)$ of $Z$ is equivalent to $\mu(dy)$. We will generally omit to mention 
$(\Omega,\mathcal{F},(\mathcal{F}_t)_{t\ge0})$ and in what follows, a process satisfying the above properties will be denoted 
by $\{Z,P_x\}$ and will simply be referred to as an $E$-valued Markov process absorbed at $\delta$. (Note that absorbtion may
or may not hold with positive probability.)\\

In all this work, we fix an integer $d\ge1$ and we denote by $\|x\|$ the Euclidean norm of $x\in\mathbb{R}^d$. We also denote by 
$S_{d-1}$ the sphere of $\mathbb{R}^d$, where $S_{d-1}=\{-1,+1\}$ if $d=1$. Let $H$ be a locally compact subspace of
$\mathbb{R}^d\setminus\{0\}$. An $H$-valued Markov process $\{X,\p_x\}$ absorbed at 0, which satisfies the following 
scaling property: there exists an index $\alpha\ge0$ such that for all $a>0$ and $x\in H$,
\begin{equation}\label{8452}
\{X,\p_{x}\}=\{(aX_{a^{-\alpha }t},\,t\ge0),\p_{a^{-1}x}\}\,,
\end{equation}
is called an $H$-valued self-similar Markov process (ssMp for short). The scaling property implies in particular that $H$ 
should satisfy $H=aH$, for any $a>0$. Therefore the space $H$ is necessarily a cone of $\mathbb{R}^d\setminus\{0\}$, that is a 
set of the form 
\begin{equation}\label{3456}
H:=\phi(S\times \mathbb{R)}\,,
\end{equation}
where $S$ is some locally compact subspace of $S_{d-1}$ and $\phi$ is the homeomorphism, 
$\phi:S_{d-1}\times \mathbb{R}\rightarrow\mathbb{R}^d\setminus\{0\}$ defined by $\phi(y,z)=ye^z$. Henceforth, 
$H$ and $S$ will be any locally compact subspaces of $\mathbb{R}^d\setminus\{0\}$ and $S_{d-1}$ respectively, 
which are related to each other by $(\ref{3456})$. The main result of this section asserts that ssMp's can be obtained as 
time changed Markov additive processes (MAP) which we now define.\\ 

A MAP $\{(\theta,\xi),P_{y,z}\}$ is an $S\times\mathbb{R}$-valued Markov process absorbed at some extra state $\delta$, 
such that for any $y\in S$, $z\in\mathbb{R}$, $s,t\ge0$, and for any positive measurable function $f$, defined on $S\times\mathbb{R}$, 
\begin{equation}\label{3473}
E_{y,z}(f(\theta_{t+s},\xi_{t+s}-\xi_t),t+s<\zeta_p\,|\,\mathcal{F}_t)=E_{\theta_t,0}(f(\theta_{s},\xi_{s}),s<\zeta_p)\ind_{\{t<\zeta_p\}}\,,
\end{equation}
where we set $\zeta_p:=\zeta(\theta,\xi)$ for the lifetime of $\{(\theta,\xi),P_{y,z}\}$, in order to avoid heavy notation. Let us now 
stress the following important remarks.\\

\begin{rem} Let $\{(\theta,\xi),P_{y,z}\}$ be any MAP and set $\theta_t=\delta'$, $t\ge\zeta_p$, for some extra state $\delta'$. 
Then according to the definition of MAP's, 
for any fixed $z\in\mathbb{R}$, the process $\{\theta,P_{y,z}\}$ is an $S$-valued Markov process absorbed at $\delta'$, such 
that $P_{y,z}(\theta_0=y)=1$, for all $y\in S$ and whose transition semigroup does not depend on $z$.
\end{rem}
\begin{rem}\label{1964} If for any $z\in\mathbb{R}$, the law of the process $(\xi_t,0\le t<\zeta_p)$ under $P_{y,z}$ does not depend 
on $y\in S$, then $(\ref{3473})$ entails that the latter is a 
possibly killed L\'evy process such that $P_{y,z}(\xi_0=z)=1$, for all $z\in\mathbb{R}$. In particular, $\zeta_p$ is exponentially 
distributed and its mean does not depend on $y,z$. Moreover, when $\zeta_p$ is finite, the process $(\xi_t,0\le t<\zeta_p)$ admits 
almost surely a left limit at its lifetime.  Examples of such MAP's can be constructed by coupling any $S$-valued Markov 
process $\theta$ with any independent real valued L\'evy process $\xi$ and by killing the couple $(\theta,\xi)$ at an independent 
exponential time. Isotropic MAP's also satisfy this property. These cases are described in Section $\ref{positive}$, see parts 
$2.$~and $3.$~of Proposition $\ref{3419}$, Definition $\ref{4362}$ and the remark which follows.
\end{rem}

MAP's taking values in general state spaces were introduced  in \cite{es} and \cite{ci}. We also refer to Chapter XI 2.a in 
\cite{as} for an account on MAP's in the case where $\theta$ is valued in a finite set. In this particular setting, they are also accurately
described in the articles \cite{kkpw} and \cite{ddk}, see Sections A.1 and A.2 in \cite{ddk}, which inspired the following extension of 
Lamperti representation. 

\begin{theorem}\label{2461}
Let  $\alpha\ge0$ and $\{(\theta,\xi),P_{y,z}\}$ be a MAP in $S\times\mathbb{R}$, with lifetime $\zeta_p$ and absorbing state $\delta$. 
Define the process $X$ by 
\[X_t=\left\{\begin{array}{lll}
\theta_{\tau_t}e^{\xi_{\tau_t}}\,,&\mbox{if}&\mbox{$t<\int_0^{\zeta_p}\exp(\alpha\xi_s)\,ds$}\,,\\
0\,,&\mbox{if}&\mbox{$t\ge\int_0^{\zeta_p}\exp(\alpha\xi_s)\,ds$}\,,
\end{array}\right.\]
where $\tau_t$ is the time change $\tau_t=\inf\{s:\int_0^se^{\alpha\xi_u}\,du>t\}$, for $t<\int_0^{\zeta_p} e^{\alpha\xi_s}\,ds$.
Define the probability measures $\p_x=P_{x/\|x\|,\log\|x\|}$, for $x\in H$ and $\p_0=P_\delta$. Then the
process $\{X,\p_x\}$ is an $H$-valued ssMp, with index $\alpha$ and lifetime $\int_0^{\zeta_p}\exp(\alpha\xi_s)\,ds$.

Conversely, let $\{X,\p_x\}$ be an $H$-valued ssMp, with index $\alpha\ge0$ and denote by $\zeta_c$ its lifetime. Define the 
process $(\theta,\xi)$ by 
\[\left\{\begin{array}{lll}
\xi_t=\log\|X\|_{A_t}&\mbox{and}\,\;\;\theta_t=\frac{X_{A_t}}{\|X\|_{A_t}}\,,\;\;\mbox{if}&t<\int_0^{\zeta_c}\frac{ds}{\|X_s\|^\alpha}\,,\\
(\xi_t,\theta_t)=\delta\,,\;\;\mbox{if}&t\ge \int_0^{\zeta_c}\frac{ds}{\|X_s\|^\alpha}\,,
\end{array}\right.\]
where $\delta$ is some extra state, and $A_t$ is the time change $A_t=\inf\{s:\int_0^s\frac{du}{\|X_u\|^\alpha}>t\}$, 
for $t<\int_0^{\zeta_c}\frac{ds}{\|X_s\|^\alpha}$. 
Define the probability measures, $P_{y,z}:=\p_{ye^z}$, for $y\in S$, $z\in\mathbb{R}$ and 
$P_\delta=\p_0$. Then the process $\{(\theta,\xi),P_{y,z}\}$ is a MAP  in $S\times\mathbb{R}$, with lifetime 
$\int_0^{\zeta_c}\frac{ds}{\|X_s\|^\alpha}$.
\end{theorem}
\begin{proof} Let $(\mathcal{G}_t)_{t\ge0}$ be the filtration corresponding to the probability space on which the MAP 
$\{(\theta,\xi),P_{y,z}\}$ is defined. Then the process $\{Y,\p_x\}$, where $\p_x$ as in the statement and 
$Y_t=\theta_{t}e^{\xi_{t}}$, for $t<\zeta_p$ and $Y_t=0$, for $t\ge\zeta_p$ is the image of $(\theta,\xi)$ through an obvious one to 
one measurable mapping, say $\phi_\delta:(S_{d-1}\times\mathbb{R})\cup\{\delta\}\rightarrow\mathbb{R}^d$. Hence it is clearly a 
standard process, as defined in the beginning of this section, in the filtration $(\mathcal{G}_t)_{t\ge0}$. Besides, if $\nu$ is the 
reference measure of $\{(\theta,\xi),P_{y,z}\}$, then $\nu\circ\phi_\delta^{-1}$ is a reference measure for $\{Y,\p_x\}$. Now define 
$\tau_t$ as in the statement if $t<\int_0^{\zeta_p} e^{\alpha\xi_s}\,ds$, and set $\tau_t=\infty$ and $X_{\tau_t}=0$, if 
$t\ge\int_0^{\zeta_p} e^{\alpha\xi_s}\,ds$. Since $e^{\xi_s}=\|Y_s\|$, $(\tau_t)_{t\ge0}$ is the right continuous inverse of 
the continuous, additive functional $t\mapsto\int_0^{t\wedge\zeta_p}\|Y_s\|^\alpha\,ds$ of $\{Y,\p_x\}$, which is strictly 
increasing on $(0,\zeta_p)$. It follows from part v of Exercise (2.11), in Chapter V of 
\cite{bg}, that $\{X,\p_x\}$ is a standard process in the filtration $(\mathcal{G}_{\tau_t})_{t\ge0}$. Finally, note that 
$\zeta_c:=\int_0^{\zeta_p} e^{\alpha\xi_s}\,ds$ is the lifetime of $X$. Then we derive from the identity 
$\e_x(\int_0^{\zeta_c} \ind_{\{X_t\in dy\}}\,dt)=\|y\|\e_x(\int_0^{\zeta_p}\ind_{\{Y_t\in dy\}}\,dt)$ and the fact that  for all
$x\in H$, $\|Y_t\|>0$, $\p_x$-a.s.,  on the set $t\in[0,{\zeta_p})>0$, that $\nu\circ\phi_\delta^{-1}$ is also 
reference measure for $\{X,\p_x\}$.

Now we check the scaling property as follows. Let $a>0$, then for $t<a^\alpha\int_0^{\zeta_p} e^{\alpha\xi_s}\,ds$,
\[\tau_{a^{-\alpha}t}=\inf\{s:\int_0^se^{\alpha(\ln a+\xi_v)}\,dv>t\}\,.\]
Let us set $\xi^{(a)}_t=\ln a+\xi_t$, then with obvious notation, $\tau_{a^{-\alpha}t}=\tau^{(a)}_t$ and
\[X_{a^{-\alpha}t}=a^{-1}\theta_{\tau_t^{(a)}}\exp(\xi^{(a)}_{\tau_t^{(a)}})\,.\]
But the equality $\{(\theta,\xi^{(a)}),P_{y,-\ln a+z}\}=\{(\theta,\xi),P_{y,z}\}$ follows from the definition (\ref{3473}) of MAP's, so that 
with $x=ye^{z}$, $a^{-1}x=ye^{-\ln a+z}$, $\p_x=P_{y,z}$ and $\p_{a^{-1}x}=P_{y,-\ln a+z}$, we have
\[\{(aX_{a^{-\alpha }t},\,t\ge0),\p_{a^{-1}x}\}=\{(X_t,\,t\ge0),\p_{x}\}\,.\]

Conversely, let $\{X,\p_x\}$ be a ssMp with index $\alpha$. Then we prove that the process $\{(\theta,\xi),P_{y,z}\}$ of the statement 
is a standard process which admits a reference measure through the same arguments as in the direct part of the proof. We only 
have to check that this process is a MAP. Let $(\mathcal{F}_t)_{t\ge0}$ be the filtration of the probability space on which $\{X,\p_x\}$ 
is defined. Define $A_t$ as in the statement if $t<\int_0^{\zeta_c}\frac{ds}{\|X_s\|^\alpha}$, set $A_t=\infty$, if 
$t\ge\int_0^{\zeta_c}\frac{ds}{\|X_s\|^\alpha}$ and note that for each $t$, $A_t$ is a stopping time of $(\mathcal{F}_t)_{t\ge0}$.Then let 
us prove that $\{(\theta,\xi),P_{y,z}\}$ is a MAP in the filtration $\mathcal{G}_t:=\mathcal{F}_{A_t}$. We denote the usual shift operator 
by $S_t$ and note that for all $s,t\ge0$, 
\[A_{t+s}=A_t+S_{A_t}(A_s)\,.\]
Set $\zeta_p=\int_0^{\zeta_c}\frac{ds}{\|X_s\|^\alpha}$.
Then from the strong Markov property of $\{X,\p_x\}$ applied at the stopping time $A_t$, we obtain from the definition of 
$\{(\theta,\xi),P_{y,z}\}$ in the statement, that for any positive, Borel function $f$, 
\begin{eqnarray*}
&&E_{\frac{x}{\|x\|},\log\|x\|}(f(\theta_{t+s},\xi_{t+s}-\xi_t),t+s<\zeta_p\,|\,\mathcal{G}_t)\\
&=&\e_{x}\left(f\left(S_{A_t}\left(\frac{X_{A_{s}}}{\|X\|_{A_{s}}}\right),\log\frac{S_{A_t}(\|X\|_{A_{s}})}{\|X\|_{A_{t}}}\right),
A_t+S_{A_t}(A_s)<\zeta_c\,|\,\mathcal{G}_t\right)\\
&=&\e_{X_{A_t}}\left(f\left(\frac{X_{A_{s}}}{\|X\|_{A_{s}}},\log\frac{\|X\|_{A_{s}}}{z}\right),
A_s<\zeta_c\right)_{z=\|X\|_{A_{t}}}\ind_{\{A_t<\zeta_c\}}\\
&=&\e_{\frac{X_{A_t}}{\|X\|_{A_{t}}}}\left(f\left(\frac{X_{A_{s}}}{\|X\|_{A_{s}}},\log\|X\|_{A_{s}}\right),
A_s<\zeta_c\right)\ind_{\{A_t<\zeta_c\}}\\
&=&E_{\theta_t,0}(f(\theta_{t+s},\xi_{t+s}-\xi_t),t+s<\zeta_p)\ind_{\{t<\zeta_p\}}\,,
\end{eqnarray*}
where the third equality follows from the self-similarity property of $\{X,\p_x\}$. We have obtained (\ref{3473}) and the theorem 
is proved.
\end{proof}
\noindent This theorem provides a one-to-one correspondence between ssMp's with index $\alpha\ge0$ and MAP's, in the general 
setting of standard processes which have a reference measure. We emphasize that $\{X,\p_x\}$ and $\{(\theta,\xi),P_{y,z}\}$ can 
have very broad behaviours at their lifetimes. For instance, $\{X,\p_x\}$ can have a finite lifetime $\zeta_c$, but may or may not 
have a left limit at $\zeta_c$. Besides whether or not $\zeta_c$ is finite, either $\int_0^{\zeta_c}\frac{ds}{\|X_s\|^\alpha}=\infty$ 
and $\{(\theta,\xi),P_{y,z}\}$ has infinite lifetime or $\int_0^{\zeta_c}\frac{ds}{\|X_s\|^\alpha}<\infty$ and $\{(\theta,\xi),P_{y,z}\}$ 
has finite lifetime $\zeta_p=\int_0^{\zeta_c}\frac{ds}{\|X_s\|^\alpha}$ and may or may not have a left limit at $\zeta_p$. 
However, in all commonly studied cases, the processes $\{X,\p_x\}$ and $\{(\theta,\xi),P_{y,z}\}$ admit almost surely a left limit 
at their lifetime.\\

For instance, if $d=1$ and $S=\{1\}$, then it follows from (\ref{3473}) that $\xi$ is a possibly killed real 
L\'evy process. Hence our result implies Theorem 4.1 of Lamperti \cite{la}, who proved that all positive self-similar Markov processes 
can be obtained as exponentials of time changed L\'evy process. In this case, in order to describe the behaviour of $\{X,\p_x\}$ at its
lifetime, it suffices to note from general properties of L\'evy processes that $\int_0^{\zeta_p}\exp(\alpha\xi_s)\,ds=\infty$ if and only if
$\xi$ is an unkilled L\'evy process such that $\limsup\xi_t=\infty$, almost surely.\\

More generally, whenever $S$ is a finite set,  for all $z$, $\{\theta,P_{y,z}\}$ is a possibly absorbed continuous time Markov chain. 
As we have already observed, the law of this Markov chain does not depend on $z$. Then it is plain from the definition that between 
two successive jump times of $\theta$, the process $\xi$ behaves like a L\'evy process. Therefore, if $n=\mbox{card}(S)$, then the 
law of $\{(\theta,\xi),P_{y,z}\}$ is characterized by the intensity matrix $Q=(q_{ij})_{i,j\in S}$ of $\theta$, $n$ non killed L\'evy processes 
$\xi^{(1)},\dots,\xi^{(n)}$, and the real valued random variables $\Delta_{ij}$, 
such that $\Delta_{ii}=0$ and where, for $i\neq j$, $\Delta_{ij}$ represents the size of the jump of $\xi$ when $\theta$ jumps from $i$ 
to $j$. More specifically, the law of $\{(\theta,\xi),P_{y,z}\}$ is given by
\begin{equation}\label{7293}
E_{i,0}(e^{u\xi_t},\theta_t=j)=(e^{A(u)t})_{i,j}\,,\;\;\;i,j\in S\,,\;\;\;u\in i\mathbb{R}\,,
\end{equation}
where $A(u)$ is the matrix,
\[A(u)=\mbox{diag}(\psi_1(u),\dots,\psi_{n}(u))+(q_{ij}G_{i,j}(u))_{i,j\in S}\,,\]
$\psi_1,\dots,\psi_{n}$ are the characteristic exponents of the L\'evy processes $\xi^{(i)}$, $i=1,\dots,n$, that is 
$E(e^{u\xi^{(i)}_1})=e^{\psi_i(u)}$, and $G_{i,j}(u)=E(\exp(u\Delta_{i,j})$. 
We refer to Sections A.1 and A.2 of \cite{ddk} for more details. We emphasize that when $d=1$, any
$\mathbb{R}\setminus\{0\}$-valued ssMp absorbed at 0 is represented by such a MAP. The case where the intensity matrix 
$Q$ is irreducible has been intensively studied in \cite{cpr}, \cite{kkpw} and \cite{ddk}.\\

We end this section with an application of Theorem \ref{2461} to a construction of ssMp's which are not killed when they hit 
0. Let $\{X,\p_x\}$ be an $\mathbb{R}^d$-valued Markov process satisfying the scaling property (\ref{8452}) with $\alpha>0$, and 
assume that it has an infinite lifetime.This means in particular that $\{X,\p_x\}$ can possibly hit 0 without being absorbed, like real 
Brownian motion for instance. Then consider the trivial real valued ssMp $\{Y,{\bf P}_y\}$, whose law is defined by 
${\bf E}_y(f(Y_{t}))=f(\mbox{sgn}(y)(|y|^{\alpha}+t)^{1/\alpha})$, for $y\in\mathbb{R}\setminus\{0\}$. The process 
$\{(X,Y),\p_x\otimes{\bf P}_y\}$ is clearly an $\mathbb{R}^{d+1}$-valued ssMp which never hits 0. 
Hence, from Theorem \ref{2461}, it 
admits a representation from a MAP $\{(\theta,\xi),P_{y,z}\}$  in $S_d\times\mathbb{R}$, such that
$\int_0^{\zeta_p}e^{\alpha\xi_s}\,ds=\infty$, $P_{y,z}$-a.s. for all $y,z$. Therefore, the process $\{X,\p_x\}$ can be represented as 
a functional of this MAP for all $t\in[0,\infty)$ and since $Y$ is deterministic, this MAP is itself a functional of $\{X,\p_x\}$.

\begin{corollary} Let $\{X,\p_x\}$ be an $\mathbb{R}^d$-valued Markov process satisfying the scaling property $(\ref{8452})$ with 
$\alpha>0$ and assume that its lifetime is infinite $\p_x$-a.s. for all $x\in\mathbb{R}^d$. Then the process
\[\left\{\begin{array}{l}
\xi_t=\log(\|X\|_{A_t}^2+A_t^{2/\alpha})^{1/2}\;\;\mbox{and}\,\;\;\theta_t=
\frac{(X_{A_t},A_t^{1/\alpha})}{(\|X\|_{A_t}^2+A_t^{2/\alpha})^{1/2}}\,,
\;\;\mbox{if}\;\;\;t<\int_0^{\infty}\frac{ds}{(\|X\|_s^2+s^{2/\alpha})^{1/2}}\,,\\
(\xi_t,\theta_t)=\delta\,,\;\;\mbox{if}\;\;\;t\ge\int_0^{\infty}\frac{ds}{(\|X\|_s^2+s^{2/\alpha})^{1/2}}\,,
\end{array}\right.\]
where $\delta$ is an extra state and $A_t=\inf\{s:\int_0^s\frac{du}{(\|X\|_u^2+u^{2/\alpha})^{\alpha/2}}>t\}$, is a MAP in $S_d\times\mathbb{R}$, with lifetime $\zeta_p$, such that $\int_0^{\zeta_p}e^{\alpha\xi_s}\,ds=\infty$, $P_{y,z}$-a.s. for all $y,z$.

Besides, the process $\{X,\p_x\}$ can be represented as follows:
\[X_t=\bar{\theta}_{\tau_t}e^{\xi_{\tau_t}}\,,\;\;\;t\ge0\,,\]
where $\tau_t=\inf\{s:\int_0^se^{\alpha\xi_u}\,du>t\}$, $\bar{\theta}=(\theta^{(1)},\dots,\theta^{(d)})$ and $\theta^{(i)}$ 
is the $i$-th coordinate of $\theta$. 
\end{corollary}

\section{Inversion and duality of ssMp's}\label{positive}

Recall that two $E$-valued Markov processes absorbed at $\delta$ with respective semigroups $(P_t)_{t\ge0}$ and $(\widehat{P}_t)_{t\ge0}$ are in weak duality with respect to some $\sigma$-finite measure $m(dx)$ if for all positive measurable 
functions $f$ and $g$, 
\begin{equation}\label{7434}
\int_{E} g(x)P_tf(x)\,m(dx)=\int_{E} f(x)\widehat{P}_tg(x)\,m(dx)\,.
\end{equation}
Duality holds when moreover, $P_t$ and $\widehat{P}_t$ are absolutely continuous with respect to $m(dx)$. However, we will make 
an abuse of language by simply saying that they are in duality, whenever they are in weak duality. With the convention that 
all measurable functions on $E$ vanish at the isolated point $\delta$, duality is sometimes defined by   
$\int_{E\cup\{\delta\}} g(x)P_tf(x)\,m(dx)=\int_{E\cup\{\delta\}} f(x)\widehat{P}_tg(x)\,m(dx)$, which is equivalent to (\ref{7434}). We 
refer to Chapter 13 in \cite{cw} where duality of standard processes is fully described. \\ 

For a MAP $\{(\theta,\xi),P_{y,z}\}$, we denote by $\{(\theta,-\xi),P_{y,-z}\}$ the process with lifetime $\zeta_p$, obtained from 
$\{(\theta,\xi),P_{y,z}\}$ simply by replacing $\xi$ by its opposite. Then $\{(\theta,-\xi),P_{y,-z}\}$ is clearly a standard process,
with an obvious reference measure, which satisfies (\ref{3473}). Hence, $\{(\theta,-\xi),P_{y,-z}\}$ is a MAP.   
In this section, we will focus on MAP's $\{(\theta,\xi),P_{y,z}\}$ such that $\{(\theta,\xi),P_{y,z}\}$  and $\{(\theta,-\xi),P_{y,-z}\}$  are 
in weak duality with respect to the measure $\pi(ds)dx$ on $S\times\mathbb{R}$, where $\pi(ds)$ is some $\sigma$-finite
measure on $S$ and $dx$ is the Lebesgue measure on $\mathbb{R}$. We will need on the following characterisation of this 
duality.

\begin{lemma}\label{9362} The MAP's $\{(\theta,\xi),P_{y,z}\}$ and $\{(\theta,-\xi),P_{y,-z}\}$ are in duality with respect to the measure
$\pi(ds)dx$ if and only if the following identity between measures
\begin{equation}\label{1363}
P_{y,0}(\theta_t\in dy_1,\xi_t\in dz)\,\pi(dy)=P_{y_1,0}(\theta_t\in dy,\xi_t\in dz)\,\pi(dy_1)\,,
\end{equation}
holds on $S\times\mathbb{R}\times S$.

We call $(\ref{1363})$ the reversibility property of the MAP $\{(\theta,\xi),P_{y,z}\}$, or equivalently we will say that 
$\{(\theta,\xi),P_{y,z}\}$ is reversible $($with respect to the measure $\pi(ds)dx$$)$.
\end{lemma}
\begin{proof} We can write for all nonnegative Borel functions $f$ and $g$ on $(S\times\mathbb{R})\cup\{\delta\}$
which vanish at $\delta$, 
\begin{eqnarray*}
&&\int_{S\times\mathbb{R}}f(y,z)E_{y,z}(g(\theta_t,\xi_t))\,\pi(dy)dz\\
&&\qquad\qquad=\int_{S}E_{y,0}\left(\int_{\mathbb{R}}f(y,z)g(\theta_t,\xi_t+z)\,dz,t<\zeta_p\right)\,\pi(dy)\\
&&\qquad\qquad=\int_{S}\int_{\mathbb{R}}E_{y,0}\left(f(y,z_1-\xi_t)g(\theta_t,z_1),t<\zeta_p\right)\,\pi(dy)\,dz_1\\
&&\quad=\int_{S^2}\int_{\mathbb{R}^2}f(y,z_1-u)g(y_1,z_1)P_{y,0}(\theta_t\in dy_1,\xi_t\in du)\,\pi(dy)\,dz_1\\
&&\quad=\int_{S^2}\int_{\mathbb{R}^2}f(y,z_1-u)g(y_1,z_1)P_{y_1,0}(\theta_t\in dy,\xi_t\in du)\,\pi(dy_1)\,dz_1\\
&&\quad=\int_{S\times\mathbb{R}}g(y_1,z_1)E_{y_1,-z_1}(f(\theta_t,-\xi_t))\,\pi(dy_1)dz_1\,,
\end{eqnarray*}
where the first identity follows from the definition of MAP's, the second one is obtained from a change of variables
and the fourth identity is due to (\ref{1363}). Then comparing the first and the last term of the above identities, we obtain
duality between $\{(\theta,\xi),P_{y,z}\}$ and $\{(\theta,-\xi),P_{y,-z}\}$ with respect to the measure $\pi(dy) dz$.
Conversely, we prove that the latter duality implies (\ref{1363}) from the same computation.
\end{proof}
\noindent By integrating (\ref{1363}) over the variable $z$, it appears from Lemma \ref{9362} that if a MAP 
$\{(\theta,\xi),P_{y,z}\}$  is reversible with respect to the measure $\pi(ds)dx$, then the Markov process $\{\theta,P_{y,z}\}$
is in duality with itself with respect to $\pi(ds)$, which is also sometimes  called the reversibility property of $\{\theta,P_{y,z}\}$ 
and justifies our terminology for $\{(\theta,\xi),P_{y,z}\}$ .\\

In the next proposition, we give sufficient conditions for a MAP to be reversible. As will be seen later on, each case 
corresponds to a well known class of ssMp's via the representation of Theorem \ref{2461}.

\begin{proposition}\label{3419} In each of the following three cases, the reversibility condition $(\ref{1363})$ is satisfied.
\begin{itemize}
\item[$1.$] Assume that $S$ is finite. Then the MAP $\{(\theta,\xi),P_{y,z}\}$ is reversible if and only if $\{\theta,P_{y,z}\}$ is in duality 
with itself with respect to some measure $(\pi_i,\,i\in S)$ and the random variables $\Delta_{ij}$ introduced in $(\ref{7293})$ are such 
that $\Delta_{ij}\ed\Delta_{ji}$, for all $i,j\in S$. 
\item[$2.$] The transition probabilities of the MAP $\{(\theta,\xi),P_{y,z}\}$ have the following particular form:
\begin{equation}\label{3492}
\left\{\begin{array}{l}
P_{y,z}(\theta_t\in dy_1,\xi_t\in dz_1)=e^{-\lambda t}P_y^{\xi'}(\theta'_t\in dy_1)P_z^{\theta'}(\xi_t'\in dz_1)\,,\\
P_{y,z}((\theta_t,\xi_t)=\delta)=1-e^{-\lambda t}\,,
\end{array}\right.
\end{equation}
for all $t\ge0$, $(y,z),(y_1,z_1)\in S\times H$, where $\lambda>0$ is some constant, $\{\xi',P_z^{\xi'}\}$ is any non killed real L\'evy 
process and $\{\theta',P_y^{\theta'}\}$ is any Markov process on $S$ with infinite lifetime. Moreover, $\{\theta',P_y^{\theta'}\}$
 is in duality with itself with respect to some $\sigma$-finite measure $\pi(ds)$.  
\item[$3.$] $S=S_{d-1}$ and for any orthogonal transformation $T$ of $S_{d-1}$,$\{(\theta,\xi),P_{y,z}\}=\{(T(\theta),\xi),P_{T^{-1}(y),z}\}$, 
for all $(y,z)\in S_{d-1}\times\mathbb{R}$. In this case, $\pi(ds)$ is the Lebesgue measure on $S_{d-1}$. $($When $d=1$, the Lebesgue
measure on $S_{d-1}$ is to be understood as the discrete symmetric measure on $\{-1,+1\}$.$)$ 
\end{itemize}
\end{proposition}
\begin{proof} 1. Recall the notation of Section \ref{sec1}. Then duality with itself of $\{\theta,P_{y,z}\}$ with 
respect to some measure  $(\pi_i,\,i\in S)$ (i.e. reversibility) is equivalent to $\pi_i q_{ij}=\pi_j q_{ji}$, for all $i,j\in S$. Since
$\Delta_{ij}\ed\Delta_{ji}$, we have $\pi_iG(u)_{ij}=\pi_jG(u)_{ji}$, which implies
\begin{equation}\label{8894}
\pi_iE_{i,0}(e^{u \xi_t},\theta_t=j)=\pi_jE_{j,0}(e^{u\xi_t},\theta_t=i)\,,
\end{equation}
and proves that $\{(\theta,\xi),P_{y,z}\}$ is reversible with respect to $(\pi_i,\,i\in S)$. Conversely, (\ref{8894}) with $u=0$ implies 
that $\{\theta,P_{y,z}\}$ is reversible with respect to  $(\pi_i,\,i\in S)$ and furthermore that $\Delta_{ij}\ed\Delta_{ji}$, for all $i,j\in S$.

2. We easily check that the law defined in (\ref{3492}) is that of a MAP. Then (\ref{1363}) follows directly from the particular form of this 
law.     
  
3. Then we prove the result in the isotropic case. Let us denote by $O(d)$ the orthogonal group and by $H(dt)$ 
the Haar measure on this group. It is known that since $O(d)$ acts transitively on the sphere $S_{d-1}$, then for any $y_1\in S_{d-1}$
and any positive Borel function $f$, defined on $S_{d-1}$,
\begin{equation}\label{5626}
\int_{O(d)}f(hy_1)H(dh)=\int_{S_{d-1}}f(y)\,dy\,.
\end{equation}
Let $g$ be another positive Borel function defined on $S_{d-1}$ and $\lambda\in\mathbb{R}$. Assume moreover that
$f(\delta)=g(\delta)=0$, then from the assumption and (\ref{5626}), for any $y_1\in S_{d-1}$, 
\begin{eqnarray*}
\int_{S_{d-1}}g(y)E_{y,0}(e^{i\lambda\xi_t}f(\theta_t))dy&=&\int_{O(d)}g(hy_1)E_{hy_1,0}(e^{i\lambda\xi_t}f(\theta_t))H(dh)\\
&=&\int_{O(d)}g(hy_1)E_{y_1,0}(e^{i\lambda\xi_t}f(h^{-1}\theta_t),t<\zeta_p)H(dh)\\
&=&E_{y_1,0}\left(e^{i\lambda\xi_t}\int_{O(d)}g(hy_1)f(h^{-1}\theta_t)H(dh),t<\zeta_p\right)\,.
\end{eqnarray*}
Let $h'\in O(d)$ such that $y_1=h'\theta_t$ and make the change of variables $h''=hh'$, then since $H(dh)$ is 
the Haar measure, we have 
\begin{eqnarray*}
&&E_{y_1,0}\left(e^{i\lambda\xi_t}\int_{O(d)}g(hy_1)f(h^{-1}\theta_t)H(dh),t<\zeta_p\right)\\
&=&E_{y_1,0}\left(e^{i\lambda\xi_t}\int_{O(d)}g(hh'\theta_t)f(h^{-1}h'^{-1}y_1)H(dh),t<\zeta_p\right)\\
&=&E_{y_1,0}\left(e^{i\lambda\xi_t}\int_{O(d)}g(h''\theta_t)f(h''^{-1}y_1)H(dh''),t<\zeta_p\right)\\
&=&\int_{S_{d-1}}f(y)E_{y,0}(e^{i\lambda\xi_t}g(\theta_t))dy\,,
\end{eqnarray*}
where we have used the assumption and (\ref{5626}) again in the last equality. Then comparing the first and the last 
term in the above equalities gives (\ref{1363}).    
\end{proof}

\noindent Examples given in this proposition lead to the definition of two important classes of ssMp's: those who satisfy the 
skew product property and those who are isotropic. 

\begin{defi}\label{4362} Let $\{(\theta,\xi),P_{y,z}\}$ and $\{X,\p_x\}$ be respectively a MAP and a ssMp which are related to 
each other through the representation of Theorem $\ref{2461}$. 

Then we say that $\{(\theta,\xi),P_{y,z}\}$ and $\{X,\p_x\}$ satisfy the skew product property if the transition probabilities of the 
MAP $\{(\theta,\xi),P_{y,z}\}$ have the form $(\ref{3492})$. $($Note that according to this definition, the process 
$\{\theta',P_y^{\theta'}\}$ involved in Proposition $\ref{3419}$ is not necessarily in duality with itself.$)$
 
We say that $\{(\theta,\xi),P_{y,z}\}$ and $\{X,\p_x\}$ are isotropic if the MAP $\{(\theta,\xi),P_{y,z}\}$ satisfies conditions of 
part $3$ of Proposition~$\ref{3419}$.
\end{defi}     

\noindent Note that this common definition of isotropy for $\{(\theta,\xi),P_{y,z}\}$ and $\{X,\p_x\}$ relies on the fact that for any 
orthogonal transformation $T$ of $S_{d-1}$, $(\theta_t,\xi_t)_{t\ge0}\ed(T(\theta_t),\xi_t)_{t\ge0}$, under $P_{y,z}$, for all 
$(y,z)\in S_{d-1}\times\mathbb{R}$ if and only if $(X_t)_{t\ge0}\ed(T(X)_t)_{t\ge0}$ under $\p_x$, for all 
$x\in\mathbb{R}^d\setminus\{0\}$. Let us also stress the following facts. \\

\begin{rem} If $\{X,\p_x\}$ is isotropic, then its norm is a positive ssMp. Hence from the Lamperti representation of $\|X\|$, the 
process $\xi$ appearing in the representation of $\{X,\p_x\}$, given in Theorem $\ref{2461}$ is a possibly killed L\'evy process. 
This fact can also be derived directly from the identity $\{(\theta,\xi),P_{y,z}\}=\{(T(\theta),\xi),P_{T^{-1}(y),z}\}$, where $T$ is any
transformation of $O(d)$ and Remark $\ref{1964}$.
\end{rem}
\begin{rem} From an obvious extension of part $3$ of Proposition $\ref{3419}$, if there exists a one-to-one measurable transformation 
$\varphi:S\rightarrow S$ such that the MAP $\{(\varphi(\theta),\xi),P_{y,z}\}$ is isotropic, then the MAP $\{(\theta,\xi),P_{y,z}\}$ is 
reversible with respect to the measure $d\varphi^{-1}(s)dx$, where $d\varphi^{-1}(s)$ is the image by $\varphi$ of the Lebesgue 
measure on $S_{d-1}$.
\end{rem}
\begin{rem} Lemma $\ref{9362}$ provides a very simple means to construct a non reversible MAP. Indeed, it suffices to consider a non 
reversible continuous time Markov chain with values in a finite set of $S_{d-1}$ and to couple it with any independent L\'evy process,
in the same way as in $(\ref{3492})$.
\end{rem}

Before stating the main result of this section, we need the following definitions. Let $\pi(dy)$ be any $\sigma$-finite measure on $S$, 
then we define the measure $\Lambda_\pi(dx)$ on $H$ as the image of the measure $\pi(dy)dz$ by the function $\phi$ defined in 
(\ref{3456}). Now let $\{\theta,\xi\}$ be a MAP which is reversible with respect to the measure $\pi(dy)dz$. As already seen in the proof
of Theorem \ref{2461}, with the probability measures $\p_x=P_{x/\|x\|,\log\|x\|}$, for $x\in H$ and 
$\p_0=P_\delta$, the process $\{Y,\p_x\}$, where $Y_t=\theta_te^{\xi_t}$, is an $H$-valued Markov
process absorbed at 0. Then following \cite{re}, see also p.240 in \cite{wa}, for $\alpha>0$, we define the measure 
$\nu_\pi^\alpha$ associated to the additive functional $t\mapsto A_t:=\int_0^t\|Y_s\|^{\alpha}\,ds$ of $\{Y,\p_x\}$ by
\[\int_H f(x)\nu_\pi^\alpha(dx)=\lim_{t\downarrow0}\int_H\e_x\left(t^{-1}\int_0^t f(Y_s)dA_s\right)\,\Lambda_\pi(dx)\,,\]
where $f$ is any positive Borel function, that is 
\begin{equation}\label{3832}
\nu_\pi^\alpha(dx)=\|x\|^\alpha \Lambda_\pi(dx)\,.
\end{equation}
In the special case where $\pi(dy)$ is absolutely continuous with density $\pi(y)$, the measure $\nu_\pi^\alpha$ can be 
explicited as 
\begin{equation}\label{3072}
\nu_\pi^\alpha(dx)=\pi(x/\|x\|)\|x\|^{\alpha-d}\,dx\,.
\end{equation}

\begin{theorem}\label{5372} Let $\{X,\p_x\}$ be a ssMp with values in $H$, with index $\alpha\ge0$ and lifetime $\zeta_c$, and 
let $\{(\theta,\xi),P_{y,z}\}$ be the MAP which is associated to $\{X,\p_x\}$ through the transformation of Theorem $\ref{2461}$.
Define the process
\begin{equation}\label{9825}
\widehat{X}_t=\left\{\begin{array}{lll}
\frac{X_{\gamma_t}}{\|X\|^2_{\gamma_t}}\,,\;\;&\mbox{if}\;\;\;t<\int_0^{\zeta_c}\frac{ds}{\|X_s\|^{2\alpha}}\,,\\
0\,,\;\;&\mbox{if}\;\;\;t\ge\int_0^{\zeta_c}\frac{ds}{\|X_s\|^{2\alpha}}\,,
\end{array}\right.
\end{equation}
where $\gamma_t$ is the time change $\gamma_t=\inf\left\{s:\int_0^s\|X\|_u^{-2\alpha}du>t\right\}$, for 
$t<\int_0^{\zeta_c}\frac{ds}{\|X_s\|^{2\alpha}}$. Define also the probability measures $\widehat{\p}_x:=\p_{x/\|x\|^2}$, for $x\in H$ and $\widehat{\p}_0:=\p_{0}$. Then the process $\{\widehat{X},\widehat{\p}_x\}$ is a ssMp with values in $H$, with index $\alpha$ and 
lifetime $\int_0^{\zeta_c}\frac{ds}{\|X_s\|^{2\alpha}}$. The MAP which is associated to $\{\widehat{X},\widehat{\p}_x\}$ through the
transformation of Theorem $\ref{2461}$ is $\{(\theta,-\xi),P_{y,-z}\}$.
Moreover, $\{X,\p_x\}$ and $\{\widehat{X},\widehat{\p}_x\}$ are in duality with 
respect to the measure $\nu_\pi^\alpha(dx)$ defined in $(\ref{3832})$ if and only if the MAP $\{(\theta,\xi),P_{y,z}\}$ is reversible with 
respect to the measure $\pi(ds)dx$. 
\end{theorem}
\begin{proof} Recall the definition of the MAP $\{(\theta,-\xi),P_{y,-z}\}$ given before Lemma \ref{9362}. Then let us define 
$\{\widehat{Y},\widehat{\p}_x\}$, where $\widehat{\p}_x$ is as in the statement and $\widehat{Y}_t:=\theta_t e^{-\xi_t}$, if $t<\zeta_p$ 
and $\widehat{Y}_t=0$, if $t\ge\zeta_p$. From the same arguments as those developed at the beginning of the proof of Theorem 
\ref{2461}, we obtain that $\{\widehat{Y},\widehat{\p}_x\}$ is a standard Markov process, which possesses a reference measure. 
Now set $\hat{\tau}_t=\inf\{s:\int_0^se^{-\alpha\xi_u}\,du>t\}$ and define
\[\left\{\begin{array}{ll}
Z_t=\theta_{\hat{\tau}_t}e^{-\xi_{\hat{\tau}_t}}\,,&\mbox{if}\;\;t<\int_0^{\zeta_p}e^{-\alpha\xi_u}\,du\,,\\
Z_t=0\,,&\mbox{if}\;\;t\ge\int_0^{\zeta_p}e^{-\alpha\xi_u}\,du\,.
\end{array}\right.\]
Since the process $\{Z,\widehat{\p}_x\}$ is constructed from the MAP $\{(\theta,-\xi),P_{y,-z}\}$ through the transformation of Theorem
$\ref{2461}$, we derive from this theorem that it is a ssMp with index $\alpha$, absorbed at 0. Moreover, it is clearly $H$-valued. 
Now let us check that for $t<\int_0^{\zeta_p}e^{-\alpha\xi_u}\,du$, $Z_t=\frac{X_{\gamma_t}}{\|X\|^2_{\gamma_t}}$. First note that from 
a change of variables, $\int_0^s\exp(-2\alpha\xi_{\tau_u})\,du=\int_0^{\tau_s}\exp(-\alpha\xi_{u})\,du$, so that with 
$\zeta_c=\int_0^{\zeta_p}\exp(\alpha\xi_{u})\,du$, we obtain 
$\int_0^{\zeta_c}\frac{ds}{\|X_s\|^{2\alpha}}=\int_0^{\zeta_p}e^{-\alpha\xi_u}\,du$. Besides it follows from the definitions that for
$t<\int_0^{\zeta_p}e^{-\alpha\xi_u}\,du$,
\begin{eqnarray*}
\gamma_t=\inf\{s,\int_0^s\exp(-2\alpha\xi_{\tau_u})\,du>t\}&=&\inf\{s,\int_0^{\tau_s}\exp(-\alpha\xi_{u})\,du>t\}\\
&=&\tau^{-1}(\hat{\tau}_t)\,, 
\end{eqnarray*}
where $\tau^{-1}$ is the right continuous inverse of $\tau$. Therefore, 
\begin{eqnarray*}
\frac{X_{\gamma_t}}{\|X\|^2_{\gamma_t}}&=&\theta_{\tau(\gamma_t)}e^{-\xi_{\tau(\gamma_t)}}\\
&=&\theta_{\hat{\tau}_t}e^{-\xi_{\hat{\tau}_t}}\,.
\end{eqnarray*}

Then assume that the MAP $\{(\theta,\xi),P_{y,z}\}$ is reversible with respect to the measure $\pi(ds)dx$.
Recall from the proof of Theorem $\ref{2461}$, the definition of the standard process $\{Y,\p_x\}$ which is constructed from 
the MAP $\{(\theta,\xi),P_{y,z}\}$. From the assumption, $\{Y,\p_x\}$ and $\{\widehat{Y},\widehat{\p}_x\}$ are in duality with 
respect to the measure $\Lambda_\pi(dx)$. Indeed, we derive from obvious changes of variables and the definition of 
$\Lambda_\pi(dx)$ that for any positive Borel functions $f$ and $g$,
\begin{eqnarray*}
\int_{S\times\mathbb{R}}f(ye^z)E_{y,z}(g(\theta_te^{\xi_t}))\,\pi(dy)\,dz=\int_{H}f(x)\e_x(g(Y_t))\,
\Lambda_\pi(dx)\,,
\end{eqnarray*}
and
\begin{eqnarray*}
\int_{S\times\mathbb{R}}g(ye^{z})E_{y,-z}(f(\theta_te^{-\xi_t}))\,\pi(dy)\,dz=\int_{H}g(x)
\widehat{\e}_x(f(\widehat{Y}_t))\Lambda_\pi(dx)\,.
\end{eqnarray*}
Since $\{(\theta,\xi),P_{y,z}\}$ is reversible, from Lemma \ref{9362} the MAP's $\{(\theta,\xi),P_{y,z}\}$ and $\{(\theta,-\xi),P_{y,-z}\}$ are 
in duality with respect to the measure $\pi(dy)\,dz$ and we derive from the above identities that 
\[\int_{H}f(x)\e_x(g(Y_t))\,\Lambda_\pi(dx)=\int_{H}g(x)\widehat{\e}_x(f(\widehat{Y}_t))\Lambda_\pi(dx)\,,\]
which proves the duality between $\{Y,\p_x\}$ and $\{\widehat{Y},\widehat{\p}_x\}$, with respect to the measure $\Lambda_\pi(dx)$.
Then from this duality and Theorem 4.5, p.~241 in \cite{wa}, we deduce that $\{X,\p_x\}$ and $\{\widehat{X},\widehat{\p}_x\}$ are 
Markov processes which are in duality with respect to the measure $\nu_\pi^\alpha(dx)=\|x\|^{\alpha}\Lambda_\pi(dx)$. 

Conversely if $\{X,\p_x\}$ and $\{\widehat{X},\widehat{\p}_x\}$ are in duality with respect to the measure $\nu_\pi^\alpha(dx)$, then
Theorem 4.5, p.~241 in \cite{wa} can be applied to $\{X,\p_x\}$ and $\{\widehat{X},\widehat{\p}_x\}$ and to the additive functionals 
$t\mapsto \int_0^{t}\frac{ds}{\|X_s\|^{\alpha}}$ and $t\mapsto\int_0^{t}\frac{ds}{\|\widehat{X}_s\|^{\alpha}}$ which are strictly increasing,
on $[0,\zeta_c)$ and $[0,\hat{\zeta}_c)$, respectively, where $\hat{\zeta}_c=\int_0^{\zeta_c}\frac{ds}{\|X_s\|^{2\alpha}}$ is the lifetime
of $\widehat{X}$. Then by reading the above arguments in reverse order, we prove that the MAP $\{(\theta,\xi),P_{y,z}\}$ is reversible 
with respect to the measure $\pi(ds)dx$.   
\end{proof}

\begin{rem} It readily follows from the previous proof that the transformation of Theorem $\ref{5372}$ is invertible, namely, 
with obvious notation: 
\[X_t=\left\{\begin{array}{lll}
\frac{\widehat{X}_{\hat{\gamma}_t}}{\|\widehat{X}\|^2_{\hat{\gamma}_t}}\,,\;\;&\mbox{if}\;\;\;t<\int_0^{\hat{\zeta}_c}\frac{ds}
{\|\widehat{X}_s\|^{2\alpha}}\,,\\
0\,,\;\;&\mbox{if}\;\;\;t\ge\int_0^{\hat{\zeta}_c}\frac{ds}{\|\widehat{X}_s\|^{2\alpha}}\,.
\end{array}\right.\]
\end{rem}
\begin{rem} In the case where $\{X,\p_x\}$ is a positive ssMp $($i.e. $d=1$ and $S=\{1\}$$)$, the fact that $\{X,\p_x\}$ is 
in duality with respect to the positive ssMp associated with the L\'evy process $-\xi$, in the Lamperti representation, was
proved in \cite{by}. It is clear that any positive ssMp satisfies the skew product property according to Definition $\ref{4362}$, so
in this case the result follows from Proposition $\ref{3419}$ and Theorem $\ref{5372}$.
\end{rem}
\begin{rem} From Proposition $\ref{3419}$, Theorem $\ref{5372}$ also applies when $\{X,\p_x\}$ is an isotropic ssMp. The existence 
of a dual process with respect to the measure $\|x\|^{\alpha-d}\,dx$ in the isotropic case was already obtained in \cite{gv}, where the 
proof relies on the wrong observation that isotropic ssMp's also satisfies the skew product property. In \cite{lw} the authors proved 
that actually if $\{(\theta,\xi),P_{y,z}\}$ is isotropic then $\theta$ and $\xi$ are independent $($i.e. $\{(\theta,\xi),P_{y,z}\}$ satisfies the 
skew product property$)$ if and only if the processes $(\xi_{\tau_t})$ and $(\theta_t)$ do not jump together, $P_{y,z}$-a.s. for all 
$y,z$. As noticed in \cite{lw}, it is not the case of isotropic stable L\'evy processes which will be studied in Section $\ref{7325}$.\\
\end{rem}

Duality of ssMp's with themselves is also an interesting property and jointly with the duality of Theorem \ref{5372} it allows us to 
obtain excessive functions, as shown in the next section. The next proposition which will be useful later on, asserts that this self 
duality holds if and only if the same property holds for the underlying MAP. 

\begin{proposition}\label{9543} 
Let $\{X,\p_x\}$ and  $\{(\theta,\xi),P_{y,z}\}$ be as in Theorem $\ref{2461}$. Then the ssMp $\{X,\p_x\}$ is in duality with  
itself with respect to some measure $M(dx)$ on $H$ if and only if the MAP $\{(\theta,\xi),P_{y,z}\}$ is in duality with itself with respect to 
the image $\eta(dy,dz)$ on $S\times \mathbb{R}$ of the measure $\|x\|^{-\alpha}M(dx)$ through the function $\phi^{-1}$, where $\phi$ is 
defined in $(\ref{3456})$.

In particular, if $M(dx)$ has a density which can be splitted as the product of an angular and a radial part, that is 
\[M(dx)=\pi(x/\|x\|)r(\|x\|)\,dx\,,\]
where $\pi$ and $r$ are nonnegative Borel functions which are respectively defined on $S$ and $\mathbb{R}_+$, then the measure 
$\eta$ on $S\times \mathbb{R}$ has the following form,
\[\eta(dy,dz)=\pi(y)e^{(d-\alpha)z}r(e^z)\,dydz\,.\] 
If moreover the MAP $\{(\theta,\xi),P_{y,z}\}$ satisfies the skew product property, then the Markov process, $\{\theta,P_{y,z}\}$ on $S$ 
is in duality with itself with respect to the measure $\pi(y)\,dy$ on $S$.
\end{proposition}
\begin{proof} Recall the notation  of the proof of Theorem \ref{5372}. Then for any
positive Borel functions $f$  and $g$, we have 
\[\int_{H}f(x)\e_x(g(X_t))\,M(dx)=\int_{H}g(x)\e_x(f(X_t))M(dx)\,.\]
Applying Theorem 4.5, p.~241 in \cite{wa}, we obtain that 
\[\int_{H}f(x)\e_x(g(Y_t))\,\|x\|^{-\alpha}M(dx)=\int_{H}g(x)\e_x(f(Y_t))\|x\|^{-\alpha}M(dx)\,,\]
which gives from a change of variables,
\begin{eqnarray*}
\int_{S\times\mathbb{R}}g(ye^{z})E_{y,z}(f(\theta_te^{\xi_t}))\,\eta(dy,dz)=
\int_{S\times\mathbb{R}}f(ye^{z})E_{y,z}(g(\theta_te^{\xi_t}))\,\eta(dy,dz).
\end{eqnarray*}
The other assertions are straightforward.
\end{proof}
 
\section{Inversion and Doob $h$-transforms for ssMp's}\label{7325}

We begin this section with a simple observation on the relationship between duality and Doob $h$-transform of Markov 
processes. 

\begin{lemma}\label{2789} Let $(P_t^{(1)})$, $(P_t^{(2)})$ and $(P^{(3)}_t)$ be the semigroups of three 
$H$-valued Markov processes absorbed at $0$. Assume that  $(P_t^{(1)})$ and $(P_t^{(2)})$ are in 
weak duality with respect to  $h_1(x)\,dx$ and that $(P_t^{(1)})$ and $(P_t^{(3)})$ are in weak duality with respect to 
$h_2(x)\,dx$, where $h_1$ and $h_2$ are positive and continuous functions. Assume moreover that $P_t^{(2)}$ and 
$P^{(3)}_t$ are Feller semigroups on $H$.

Then $h:=h_1/h_2$ is excessive for $(P_t^{(3)})$ and the Markov process with semigroup $(P_t^{(2)})$ is an $h$-process of  the
Markov process with semigroup $(P_t^{(3)})$, with respect to the function $h$, that is 
\begin{equation}\label{4350}
P_t^{(2)}g(x)=\frac{1}{h(x)}P_t^{(3)}(hg)(x)\,,
\end{equation}
for all $t\ge0$, $x\in H$ and  all positive Borel functions $g$. 

Conversely, if $(\ref{4350})$ holds and if $(P_t^{(1)})$ and $(P_t^{(2)})$ are in weak duality with respect to  $h_1(x)\,dx$, then 
$(P_t^{(1)})$ and $(P_t^{(3)})$ are in weak duality with respect to $h_2(x)\,dx$.
\end{lemma}
\begin{proof} From the assumption, for all positive Borel functions $f$ and $g$, 
\begin{eqnarray}
&&\int_{H} P_t^{(1)}f(x) g(x)\,h_1(x)dx=\int_{H}f(x)P_t^{(2)}g(x)\,h_1(x)dx\label{1958}\,,\\
&&\int_{H} P_t^{(1)}f(x) g(x)\,h_2(x)dx=\int_{H}f(x)P_t^{(3)}g(x)\,h_2(x)dx\nonumber\,.
\end{eqnarray}
Replacing $g$ by $\frac{h_1}{h_2}g$ in both members of the second equality gives for all $f$,
\begin{equation}\label{9562}
\int_{H} P_t^{(1)}f(x) g(x)\,h_1(x)dx=\int_{H}f(x)\frac{h_2}{h_1}(x)P_t^{(3)}\frac{h_1}{h_2}g(x)\,h_1(x)dx\,.
\end{equation}
Identifying the second members of (\ref{1958}) and (\ref{9562}), yields (\ref{4350}) in the case where $g$ and $\frac{h_1}{h_2}g$
are bounded and continuous, thanks to the Feller property. Then we extend (\ref{4350}) to all positive Borel functions from 
classical arguments.

The converse is proved in the same way.
\end{proof}

Then as an application of this lemma and Theorem \ref{5372} we can recover the results recalled in the introduction, that
is Theorem 1 in \cite{bz} and Theorem 4 in \cite{ky}.

\begin{corollary}\label{1511}
Let $\{X,\p_x\}$ be a $d$-dimensional stable L\'evy process, with index $\alpha\in(0,2]$, which is absorbed at its first hitting time 
of $0$ $($absorbtion actually holds with probability $1$ if $d=1$ and $\alpha>1$, and with probability $0$ in the other cases$)$.
Recall the definition of the process $\{\widehat{X},\widehat{\p}_x\}$, from Theorem $\ref{5372}$.
\begin{itemize}
\item[$1.$]
If $d=1$ and if $\{X,\p_x\}$ is not spectrally one sided, then the process $\{\widehat{X},\widehat{\p}_x\}$ is an $h$-process of  
$\{-X,\p_{-x}\}$ with respect to the function $x\mapsto\pi(x/|x|)|x|^{\alpha-1}$, where $(\pi(-1),\pi(+1))$ is the invariant measure of the 
first coordinate of the MAP associated to $\{X,\p_x\}$.
\item[$2.$] If $d>1$ and $\{X,\p_x\}$ is isotropic, then the process $\{\widehat{X},\widehat{\p}_x\}$ is an $h$-process of  $(X,\p_{x})$ 
with respect to the function $x\mapsto\|x\|^{\alpha-d}$.
\end{itemize}
\end{corollary}
\begin{proof}  When $d=1$ and $\{X,\p_x\}$ is not spectrally one sided, we derive from Corollary 11, Section 4.1 in \cite{cpr}, that the 
MAP associated to the stable L\'evy process $\{X,\p_x\}$ satisfies $\Delta_{ij}\ed\Delta_{ji}$, $i,j\in\{-1,+1\}$, with the notation introduced 
in (\ref{7293}). Moreover, since the continuous time Markov chain $\{\theta,P_{y,z}\}$ is irreducible and takes only two values, it is 
reversible with respect to some measure $(\pi(-1),\pi(+1))$. Then we derive from part 1.~of Proposition \ref{3419} that the MAP 
$\{(\theta,\xi),P_{y,z}\}$  is reversible. Hence we can apply Theorem \ref{5372} which ensures that $\{X,\p_x\}$ and 
$\{\widehat{X},\widehat{\p}_x\}$ are in duality with respect to the function $x\mapsto\pi(x/|x|)|x|^{\alpha-1}$. 

If $d\ge1$ and the process is isotropic, then again we derive from Proposition \ref{3419} and Theorem \ref{5372} that  $\{X,\p_x\}$ 
and $\{\widehat{X},\widehat{\p}_x\}$ are in duality with respect to the function $x\mapsto\|x\|^{\alpha-d}$. 

Finally it is well known that stable L\'evy processes satisfy the Feller property. Then from the path construction in Theorem \ref{5372}
and homogeneity of the increments of $\{X,\p_x\}$, we derive that the process $\{\widehat{X},\widehat{\p}_x\}$ is itself a Feller process 
on $\mathbb{R}^d\setminus\{0\}$. Moreover it is well known  that $\{X,\p_x\}$ and $\{-X,\p_{-x}\}$ are in duality with respect to the 
Lebesgue measure (when $d=1$ and $\alpha>1$, this duality is inherited from the duality between the non absorbed L\'evy processes 
with respect to the Lebesgue measure on $\mathbb{R}$).  It remains to apply Lemma \ref{2789} in order to conclude in both cases 
$d=1$ and $d>1$. 
\end{proof}
\noindent Note that in the case $d=1$, the probability measure $(\pi(-1),\pi(+1))$ has been made explicit in Corollary 11 of \cite{cpr} 
and in Theorem 4 of \cite{ky}. Besides, the $h$-process involved in part 1 of Corollary \ref{1511} has been extensively investigated in
\cite{ya} and \cite{pa}, where for $\alpha>1$, it is identified as the real L\'evy process $\{-X,\p_{-x}\}$ conditioned to avoid 0.\\

Theorem \ref{5372} and Lemma \ref{2789} can be applied in the same way as in Corollary \ref{1511}, to any reversible ssMp 
$\{X,\p_x\}$ provided $\{\widehat{X},\widehat{\p}_x\}$ has the Feller property on $H$ and $(X,\p_x)$ is in duality with some other 
Feller process on $H$. Let us give three examples.\\

\noindent {\bf A.~Conditioned L\'evy processes}: Let $\{X,\p_x\}$ be a one dimensional stable L\'evy process, with index $\alpha\in(0,2)$ 
and let us denote by $X^0$ the process $X$ which is absorbed at 0 when it first hits the negative halfline, that is 
\[X^0_t=X_t\ind_{\{t<\tau^-\}}\,,\;\;\;t\ge0\,,\] 
where $\tau^-=\inf\{t:X_t<0\}$. It is well known, see \cite{cc} and the references therein, that the functions $h_1(x)=x^{\alpha(1-\rho)}$ 
and $h_2(x)=x^{\alpha(1-\rho)-1}$, where $\rho=\p_0(X_1>0)$, are respectively invariant and excessive for the process $(X^0,\p_x)$. 
We denote the Doob $h$-transforms of $(X^0,\p_x)$ associated to $h_1$ and $h_2$, respectively by 
\[\{X^\uparrow,\p_x^\uparrow\}\;\;\;\mbox{and}\;\;\;\{X^{\searrow},\p_x^\searrow\}\,.\]
The process $\{X^\uparrow,\p_x^\uparrow\}$ is known as the L\'evy process $\{X,\p_x\}$ conditioned to stay positive. It is a positive 
ssMp with index $\alpha$, which satisfies $\lim_{t\rightarrow\infty}X_t=+\infty$, $\p_x$-a.s., for all $x>0$. In particular this process 
never hits 0. In the case of Brownian motion, that is for $\alpha=2$, it corresponds to the three dimensional Bessel process. The 
process $\{X^{\searrow},\p_x^\searrow\}$ is called the process $\{X,\p_x\}$ conditioned to hit 0 continuously. It is also a positive ssMp 
with index $\alpha$. Its absorbtion time $\zeta_c$ is finite and satisfies $X_{\zeta_c-}=0$, $\p_x^\searrow$-a.s. for all $x>0$.  Note
that in the present case, $H=(0,\infty)$. 

Then let us set $Y:=-X$ and denote by ${\rm P}_x$, $x\in\mathbb{R}$, the family of probability measures associated to $Y$. 
We  denote by $\{Y^\uparrow,{\rm P}_x^\uparrow\}$ and $\{Y^{\searrow},{\rm P}_x^\searrow\}$ the process $\{Y,{\rm P}_x\}$ 
conditioned to stay positive and conditioned to hit 0 continuously, respectively. Since $\{Y,{\rm P}_x\}$ is a stable L\'evy process, these
processes enjoy the same properties as $\{X^\uparrow,\p_x^\uparrow\}$ and $\{X^{\searrow},\p_x^\searrow\}$. 

The process $\{X,\p_x\}$ conditioned to stay positive is related to the process $\{Y,{\rm P}_x\}$ conditioned to hit 0 continuously 
as follows. 

\begin{corollary}\label{4568} The process $\{X^\uparrow,\p_x^\uparrow\}$ can by obtained from the paths of the process 
$\{Y^{\searrow},{\rm P}_x^\searrow\}$, through the following transformation$:$
\[\{X^\uparrow,\p_x^\uparrow\}=\{(1/Y_{\gamma_t}^{\searrow})_{t\ge0},{\rm P}_{1/x}^\searrow\}\,,\]
where the time change $\gamma_t$ is defined by 
\[\gamma_t=\inf\left\{s:\int_0^s\frac{du}{(Y^\searrow_u)^{2\alpha}}>t\right\}\,,\;\;\;t\ge0\,.\]
\end{corollary}
\begin{proof} As a positive ssMp, the process $\{Y^{\searrow},{\rm P}_x^\searrow\}$ satisfies the skew product property. 
Therefore it is reversible and from Theorem \ref{5372}, $\{Y^{\searrow},{\rm P}_x^\searrow\}$ and 
$\{(1/Y_{\gamma_t}^{\searrow})_{t\ge0},{\rm P}_{1/x}^\searrow\}$ are in duality with respect to the measure $x^{\alpha-1}\,dx$.
Moreover, $\{Y^{\searrow},{\rm P}_x^\searrow\}$ and $\{X^\uparrow,\p_x^\uparrow\}$ are also in duality with respect to the measure 
$x^{\alpha-1}\,dx$. Indeed, let us denote by $Y^0$ the process $Y$ absorbed at 0 when it first hits the negative halfline. Then 
$\{X^0,\p_x\}$ and $\{Y^0,{\rm P}_x\}$ are in duality with respect to the Lebesgue measure on $(0,\infty)$ (this duality is inherited 
from the duality between the L\'evy processes $\{X,\p_x\}$ and $\{Y,{\rm P}_x\}$ with respect to the Lebesgue measure on 
$\mathbb{R}$). Therefore, for any positive Borel functions $f$ and $g$,
\begin{eqnarray*}
\int_0^\infty g(x){\rm E}_x^\searrow(f(Y^\searrow_t))\,x^{\alpha-1}\,dx&=&\int_0^\infty g(x){\rm E}_x(Y_t^{\alpha\rho-1}f(Y_t),
t<\tau_Y^-)\,x^{\alpha(1-\rho)}\,dx\\
&=&\int_0^\infty f(x)\e_x(X_t^{\alpha(1-\rho)}g(X_t),t<\tau^-)\,x^{\alpha\rho-1}\,dx\\
&=&\int_0^\infty f(x)\e_x^\uparrow(g(X_t^\uparrow))\,x^{\alpha-1}\,dx\,,
\end{eqnarray*}
where $\tau_Y^-=\inf\{t:Y_t<0\}$ in right hand side of the above equality. 

Then as positive ssMp's, $\{X^\uparrow,\p_x^\uparrow\}$ and $\{(1/Y_{\gamma_t}^{\searrow})_{t\ge0},{\rm P}_{1/x}^\searrow\}$ 
satisfy the Feller property on $(0,\infty)$, see Theorem 2.1 in \cite{la}. Therefore applying Lemma \ref{2789}, we conclude that
$\{X^\uparrow,\p_x^\uparrow\}$ is an $h$ process of the process $\{(1/Y_{\gamma_t}^{\searrow})_{t\ge0},{\rm P}_{1/x}^\searrow\}$ 
with respect to the function which is identically equal to 1, so that both processes are equal.
\end{proof}
\noindent It is straightforward that the same relationship exists between the processes $\{Y^\uparrow,{\rm P}_x^\uparrow\}$ and
$\{X^{\searrow},\p_x^\searrow\}$. Let us also note that when $\alpha>1$ and $\{X,\p_x\}$ has no positive jumps, then 
$\{Y^{\searrow},{\rm P}_x^\searrow\}=\{Y^0,{\rm P}_x\}$. Therefore Corollary \ref{4568} and its proof allow us to complete 
part 1 of Corollary \ref{1511}, where completely asymmetric L\'evy processes are excluded, by stating that 
$\{\widehat{X},\widehat{\p}_x\}$ is an $h$-process of $\{-X,\p_{-x}\}$ with respect to the function $x\mapsto x^{\alpha-1}$.
We emphasize that here we consider $\{\widehat{X},\widehat{\p}_x\}$ and $\{-X,\p_{-x}\}$ as positive ssMp's, that is $H=(0,\infty)$. 
Then in this case, $\{\widehat{X},\widehat{\p}_x\}$ corresponds to the process  $\{-X,\p_{-x}\}$ conditioned to 
stay positive. This remark also allows us to recover Theorem \ref{6214}. \\

\noindent {\bf B.~Free $d$-dimensional Bessel processes}: Let $X=\{(X_1(t),X_2(t),..,X_d(t)),t\ge0\}$, where $X_i(t)$ are independent
BES$(\delta)$ processes of dimension $\delta>0$ and let us consider the $H$-valued ssMp $\{X,\p_x\}$ absorbed at 0, where
$H=(0,\infty)^d$. It is well known that for each $i=1,\dots,d$, $X_i$ is in duality with itself with respect to its speed measure
$m_i(dx_i)= x_i^{1+2\nu}/|\nu|dx_i$ when $\nu\neq 0$ and $m_i(dx_i)= x_idx_i$ when $\nu=0$, where $\nu=\frac{\delta}{2}-1$.
This entails that $\{X,\p_x\}$ is in duality with itself, with respect to the measure $D(x)\,dx$, where
\begin{equation}\label{product}
D(x)=\prod_{i=1}^dx_i^{\delta-1} =||x||^{d\delta-d} \prod_{i=1}^d \left(\frac{x_i}{||x||}\right)^{\delta-1}\,,\;\;\;
x\in H\,.
\end{equation}
Let us observe that the process $\{X,\p_x\}$ satisfies the skew product property.  Indeed we can argue similarly as for the skew
product decomposition of Brownian motion, see e.g. Chapter 7.15 in \cite{IMcK}. The generator of $\{X,\p_x\}$ is equal to
$L=\sum_{i=1}^d L_{x_i}$, where   $L_{x_i}=\frac12\frac{\partial^2}{\partial x_i^2}+\frac12\frac{\delta-1}{x_i}\frac{\partial}{\partial x_i}$.
Then we compute the spherical decomposition of $L$. Set $u=u(r,\sigma)$, where $r=\|x\|$ and $\sigma=\frac{x}{\|x\|}$
are the spherical coordinates in $\mathbb{R}^d$ and let $\Delta_\sigma$ be the Laplacian on the unit sphere $S_{d-1}$.
By the well-known formula $\Delta =\frac12\frac{\partial^2 }{\partial r^2} + \frac12\frac{ d -1}{r}\frac{ \partial }{\partial r}+\frac{1}{r^2}
\Delta_{\sigma} $  and by the chain rule, we obtain
$$L u = \frac12\frac{\partial^2 u}{\partial r^2} + \frac12\frac{\delta d -1}{r}\frac{ \partial u}{\partial r}+\frac{1}{r^2}
\left[ \Delta_{\sigma} u +\frac{\delta -1}{2} (\sigma^{-1}-d \sigma) \cdot \nabla_\sigma u \right],$$
where $(\sigma^{-1})_i=\sigma^{-1}_i$.
We notice that $L_r:=\frac12\frac{\partial^2 u}{\partial r^2} + \frac12 \frac{\delta d -1}{r}\frac{ \partial u}{\partial r}$
is the generator of $\|X\|\sim$BES$(\delta d)$.\\

We deduce from Propositions \ref{3419} and \ref{9543} that the free Bessel process $\{X,\p_x\}$ is reversible, with respect to
the measure with density $\pi(y)= \prod_{i=1}^d y_i^{\delta-1},\ \ y\in S_{d-1}\cap H$. Indeed, it satisfies the skew product property and 
it is self dual with respect to the measure $D(x)dx$ which splits as the product of an angular part and a radial part, see \eqref{product}. 
Therefore Lemma \ref{2789} and Theorem \ref{5372} can be applied to the free Bessel process $\{X,\p_x\}$ and we obtain the 
following result.

\begin{corollary}\label{1698}
Let $\{X,\p_x\}$ be a free Bessel process with values in $(0,\infty)^d$ and absorbed at $0$. Recall the definition of the process 
$\{\widehat{X},\widehat{\p}_x\}$, from Theorem $\ref{5372}$. Then the processes $\{X,\px\}$ and $\{\widehat{X},\widehat{\p}\}$ are 
in duality with respect to the measure with density
$$\prod_{i=1}^d \left(\frac{x_i}{\|x\|}\right)^{\delta-1} \|x\|^{2-d}.$$
Moreover, the process $\{\widehat{X},\widehat{\p}_x\}$ is a Doob $h$-transform of $\{X,\px\}$ with respect to the excessive
function $h(x)= \|x\|^{2-d\delta}$.
\end{corollary}

\begin{rem}
It is easy to check that for $x\in H$ one has $Lh=0$, i.e. $h$ is $L$-harmonic on its domain. However, $h$
is not $X$-invariant when $d\delta>2$, i.e. $(||X_t||^{2-d\delta}, t\geq 0)$ is a strict local martingale. See for instance the
discussion on p.~$330$ of \cite{Elworthy-Li-Yor}.
\end{rem}
\begin{rem} For $\delta>2$, we can give the following realization of the MAP corresponding to the free Bessel process
$\{X,\p_x\}$. There exists a $d$-dimensional Brownian motion $(W^{(1)}, W^{(2)}, \cdots, W^{(d)})$ such that
$$\xi_t=\sum_{j=1}^d\int_0^t \theta^{(j)}_s dW^{(j)}_s +\left(\frac{d\delta}{2}-1\right)t$$ and $(\theta_t)$ satisfies the SDE system
\begin{eqnarray*}
d\theta^{(i)}_t=dW^{(i)}_t-\theta^{(i)}_t\sum_{j=1}^d \theta_t^{(j)} dW^{(j)}_t+ \left(\frac{\delta-1}{2}\frac{1}{\theta_t^{(i)}}-\frac{d\delta-1}{2}\theta_t^{(i)} \right)dt,
\quad i=1,2,\ldots, d.
\end{eqnarray*}
The processes $(\xi_t)$ and $(\theta_t)$ are independent. The proof of this MAP representation relies on the SDE representations 
of the processes $X_i(t)$.\\
\end{rem}

\noindent {\bf C.~Dunkl processes}: In what follows, we recall and use  some properties of the Dunkl processes which may be found
in Chapters 2 and 3 of \cite{CGYor}.

Let $R$ be a finite root system on $\mathbb{R}^d$ and let $R^+$ be a positive subsystem of $R$. Let also $k$ be a non-negative
function on $R$, called {\it multiplicity function}. The generator of the Dunkl process $\{X,\p_x\}$ is $L_k:=\frac12 \Delta_k$ where $\Delta_k=\sum_{i=1}^d T_i^2$ is the Dunkl Laplacian  and $T_if(x):= \partial_i f(x) + 
\sum_{ \alpha \in R^+} k(\alpha) \alpha_i \frac{f(x)-f(\sigma_\alpha x)}{\alpha\cdot x}$, $i=1,2, \cdots, d$, are  the Dunkl derivatives.

Here $\sigma_\alpha$ are the symmetries with respect to the hyperplanes $\{\alpha=0\}$. Dunkl processes are ssMp's with
index $2$. In Dunkl  analysis, an important role is played by the so-called Dunkl weight function $\omega_k(x)=
\prod_{ \alpha \in R} |\alpha\cdot x|^{k(\alpha)} $ and the constant $\gamma=\gamma(k)=\sum_{ \alpha \in R^+} k(\alpha).$
We see that $\omega_k(x)$ is  homogeneous  of order $2\gamma$.

Let us also mention that Dunkl processes have the skew product property: this fact was proved by Chybiryakov in \cite{CGYor}, see
Theorem 8, p.156 therein. Besides, the radial part $R_t=\|X_t\|$ is a BES$(d+2\gamma) $ process.

We also observe that the Dunkl process $\{X,\p_x\}$ is self-dual  with respect to the measure $M(dx)= \omega_k(x)\,dx$. This follows 
from the formula for the Dunkl  transition function, see (23) p.120 in \cite{CGYor},
\begin{equation}\label{DunklHEAT}
p^{(k)}_t(x,y)=\frac{1}{c_k t^{\gamma+d/2}} \exp\left(-\frac{\|x\|^2+\|y\|^2}{2t}\right)D_k\left(\frac{x}{\sqrt{t}},\frac{y}{\sqrt{t}}\right)
\omega_k(y)\,,
\end{equation}
where $D_k$ is the Dunkl kernel.  The only non-symmetric factor in \eqref{DunklHEAT} is $\omega_k(y)$; hence  the kernel
$p^{(k)}_t(x,y)  \omega_k(x)$ is symmetric in $x$ and $y$.

The density of the self-duality measure $M$ factorizes as
$$  \omega_k(x)=  \omega_k(x/\|x\|) \|x\|^{2\gamma}. $$
By Proposition \ref{9543},  the Dunkl process $\{X,\p_x\}$ is reversible with respect to the measure $\pi(y)=\omega_k(y),\ y\in S_{d-1}$.

Note that contrary to the case of Brownian motion, the Dunkl process $\{X,\p_x\}$ with $k\not=0$ is non-isotropic.
Indeed, the process $\{X,\p_x\}$ always jumps from a state $y$ to a symmetric state $\sigma_\alpha(y)$. Thus, like free Bessel 
processes, Dunkl processes are a class of reversible non-isotropic self-similar processes. Then we derive the next corollary as a
consequence of Theorem \ref{5372}.

\begin{corollary}\label{2352}
Let $\{X,\p_x\}$ be a Dunkl process in $\mathbb{R}^d\setminus\{0\}$ and absorbed at $0$.  Recall the definition of the process
$\{\widehat{X},\widehat{\p}_x\}$, from Theorem $\ref{5372}$. The processes $\{X,\px\}$ and $\{\widehat{X},\widehat{\p}_x\}$ are in 
duality with respect to the measure with density
$$ \omega_k(x/\|x\|)\|x\|^{2-d}.$$
Moreover, the process $\{\widehat{X},\widehat{\p}_x\}$ is a Doob $h$-transform of $\{X,\px\}$ with respect to the excessive
function $h(x)= \|x\|^{2-d-2\gamma}$.
\end{corollary}
\begin{rem} The function $h$ of Corollary $\ref{2352}$ is always Dunkl-harmonic in the sense that $\Delta_k h = 0$. This follows from 
the form  of the Dunkl Laplacian in polar coordinates, see \cite{Xu}. This is confirmed by the well-known fact that $(h(X_t), t\leq \zeta_c)$ 
is a local martingale which is a true martingale only when $d+2\gamma\leq 2$.
\end{rem}


\end{document}